\documentclass[11pt,reqno]{amsart}

\usepackage[a4paper,margin=1.12in]{geometry}
\usepackage{amsmath,amssymb,amsthm,mathtools}
\usepackage{enumitem}
\usepackage{xcolor}
\usepackage{microtype}
\usepackage{hyperref}
\usepackage[nameinlink,capitalize,noabbrev]{cleveref}

\hypersetup{
  colorlinks=true,
  linkcolor=blue!55!black,
  citecolor=blue!55!black,
  urlcolor=blue!55!black
}

\newtheorem{theorem}{Theorem}[section]
\newtheorem{proposition}[theorem]{Proposition}
\newtheorem{lemma}[theorem]{Lemma}
\newtheorem{corollary}[theorem]{Corollary}

\theoremstyle{definition}

\newtheorem{problem}[theorem]{Problem}
\theoremstyle{remark}
\newtheorem{remark}[theorem]{Remark}
\newtheorem{example}[theorem]{Example}

\newcommand{\C}{\mathbb C}
\newcommand{\R}{\mathbb R}
\newcommand{\Q}{\mathbb Q}

\newcommand{\B}{\mathbb B}

\newcommand{\Jac}{J_{\!\R}}
\newcommand{\Repart}{\operatorname{Re}}

\newcommand{\Crit}{\operatorname{Crit}_{\R}}

\newcommand{\wtH}{\widetilde H}
\newcommand{\m}{\mathfrak m}

\title[An All-Dimensional Lewy Theorem]
{An all-dimensional Lewy theorem for pluriharmonic mappings via Milnor monodromy} 

\author[D. Zhong and Z.-G. Wang]{Deguang Zhong and Zhi-Gang Wang$^*$}

\address{\noindent Deguang Zhong\vskip.05in
Institute of Applied Mathematics, Shenzhen Polytechnic
University, Shenzhen 518055, Guangdong, P.~R. China.}
\email{zhongdg1014$@$szpu.edu.cn}

\address{\noindent Zhi-Gang Wang\vskip.05in
School of Mathematics and Statistics, and Hunan Provincial University Key Laboratory for Big Data Analysis and Application, Hunan First Normal University,
Changsha 410205, Hunan, P. R. China.}
\email{sjyzhigangwang$@$hnfnu.edu.cn}
\thanks{$^*$Corresponding author.}


\begin{document}

\begin{abstract}
Hengartner asked whether a pluriharmonic mapping between equidimensional
complex Euclidean domains is locally one-to-one precisely at the points where
its real Jacobian is nonzero.  Naser had earlier claimed the conclusion in
complex dimension two, and a recent theorem of Kalaj supplies a complete proof
in that dimension.  We prove the result in every complex dimension.  The main
ingredient is a singularity-theoretic characterization which is of independent
interest: if $q\colon(\C^n,0)\to(\C,0)$ is a nonconstant holomorphic germ, then
$dq(0)\ne0$ if and only if one (equivalently, every) phase hypersurface
\(
   \{\Repart(e^{-i\theta}q)=0\}
\)
is a rational homology $(2n-1)$-manifold at the origin.  The difficult
direction combines local conical structure, Alexander duality, the Milnor
fibration, and A'Campo's vanishing theorem for the Lefschetz number of local
monodromy.  The argument requires neither an isolated critical point nor a
reducedness hypothesis.  It follows that the real critical set of any
pluriharmonic self-dimensional mapping agrees exactly with its local
noninjectivity set.  As a quantitative application, combining our theorem with
Kalaj's boundary regularity theorem shows that, in every complex dimension, a
quasiconformal pluriharmonic homeomorphism of the unit ball onto a bounded
$C^1$-Dini domain is bi-Lipschitz.
\end{abstract}

\maketitle

\section{Introduction and statement of the main result}

Lewy's classical theorem asserts that a one-to-one harmonic mapping between
planar domains has nowhere-vanishing Jacobian \cite{Lewy1936}.  Consequently,
planar harmonic univalence automatically upgrades to local real-analytic
diffeomorphism.  This rigidity is exceptional among general harmonic maps.
Wood constructed a harmonic homeomorphism of $\R^3$ whose Jacobian vanishes
\cite{Wood1991}, and related failures of injectivity for higher-dimensional
harmonic extensions were exhibited by Laugesen \cite{Laugesen1996}.  Positive
higher-dimensional results therefore require additional structure; compare
Lewy's theorem for harmonic gradients \cite{Lewy1968} and later work on
harmonic quasiconformal mappings \cite{AstalaManojlovic2015,
BozinMateljevic2015}.

Pluriharmonic mappings form a natural intermediate class.  A complex-valued
$C^2$ function $f$ on a domain in $\C^n$ is pluriharmonic when
\[
   \frac{\partial^2 f}{\partial z_j\,\partial\overline z_k}=0,
   \quad 1\le j,k\le n.
\]
Equivalently, $f$ is locally of the form $h+\overline g$ with $h$ and $g$
holomorphic.  In particular, every real-valued pluriharmonic function is
locally the real part of a holomorphic function.  We refer to
\cite{GunningRossi1965,Hormander1990,Rudin1980} for the several-variable
background and to \cite{ChenGauthier2011,HamadaKohr2015} for aspects of the
function theory of pluriharmonic mappings.

The hypersurfaces $\{\Repart q=0\}$ also form a basic class of singular
real-analytic Levi-flat hypersurfaces: away from the critical locus of $q$,
their Levi foliation is given by the complex level sets of $q$.  Singular
Levi-flat hypersurfaces have a substantial local theory
\cite{Brunella2007,BurnsGong1999}.  From that viewpoint,
Theorem~\ref{thm:phase-criterion} says that a singularity coming from a holomorphic
first integral is always visible already in local rational homology.

In 1973, Naser \cite{Naser1973} stated a four-real-dimensional pluriharmonic analogue of Lewy's
theorem.  The general question was recorded by Bshouty and
Lyzzaik \cite{BshoutyLyzzaik2010} as Problem~3.10(a), attributed to Hengartner.

\begin{problem}\label{q:hengartner}
{\it Let $f_1,\ldots,f_n$ be complex-valued pluriharmonic functions on the unit
ball $B\subset\C^n$, and put $F=(f_1,\ldots,f_n)$.  Is $F$ locally univalent
if and only if its real Jacobian does not vanish?}
\end{problem}

Kalaj \cite{Kalaj2026} recently identified a local-topological gap in the earlier
four-dimensional argument and proved the result for maps
$\C^2\to\C^2$.  His proof reduces degeneracy to a critical
holomorphic germ $q\colon(\C^2,0)\to(\C,0)$ and then studies the topology of
$\{\Repart q=0\}$ through plane-curve links.  That link analysis is specific
to complex dimension two.

The purpose of this paper is to replace the dimension-specific part by a
general monodromy obstruction.  Our first theorem is stronger than the local
flatness statement required for the mapping problem.  For a locally compact
space $X$ and $x\in X$, we use $H_*(X,X\setminus\{x\};\Q)$ for local homology,
understood on a sufficiently small representative when $X$ is a germ.

\begin{theorem}
\label{thm:phase-criterion}
Let $q\colon(\C^n,0)\to(\C,0)$ be a nonconstant holomorphic germ, $n\ge1$.
For $\theta\in\R$, set
\[
   X_\theta(q):=\{z:\Repart(e^{-i\theta}q(z))=0\}.
\]
The following conditions are equivalent:
\begin{enumerate}[label=\textup{(\roman*)}]
\item $dq(0)\ne0$;
\item for every $\theta\in\R$, the germ $X_\theta(q)$ is a smooth
      real-analytic hypersurface at $0$;
\item for some $\theta\in\R$, the local rational homology of $X_\theta(q)$
      at $0$ is that of $\R^{2n-1}$, namely
\begin{equation}\label{eq:local-Q-homology}
H_j\bigl(X_\theta(q),X_\theta(q)\setminus\{0\};\Q\bigr)
 \cong
 \begin{cases}
   \Q,&j=2n-1,\\
   0,&j\ne2n-1.
 \end{cases}
\end{equation}
\end{enumerate}
Thus, at a holomorphic critical point, no phase hypersurface $X_\theta(q)$ is
even a rational homology manifold of the expected dimension.
\end{theorem}

The use of rational rather than integral local homology is deliberate.  It is
the weakest standard manifold-type hypothesis needed by the proof, since the
Lefschetz number is computed on rational homology.  In particular,
Theorem~\ref{thm:phase-criterion} rules out topological-manifold and locally flat
hypersurface germs a fortiori.

Our answer to Problem \ref{q:hengartner} is the following pointwise theorem.

\begin{theorem}
\label{thm:lewy}
Let $\Omega\subset\C^n$ be a domain, $n\ge1$, and let
$F\colon\Omega\to\C^n$ be pluriharmonic.  For every $p\in\Omega$, the
following are equivalent:
\begin{enumerate}[label=\textup{(\roman*)}]
\item $F$ is one-to-one on a neighborhood of $p$;
\item $\Jac F(p)\ne0$;
\item $F$ is a local real-analytic diffeomorphism at $p$.
\end{enumerate}
Consequently, Hengartner's Problem~3.10(a) has an affirmative answer in every
complex dimension.
\end{theorem}

There are two further consequences worth emphasizing.  First, the theorem
identifies, rather than merely compares, two a priori different singular sets:
\[
 \{\Jac F=0\}
 =\{p:F\text{ is not one-to-one on any neighborhood of }p\}.
\]
This yields a real-analytic branch-set dichotomy; see
Theorem~\ref{thm:branch-identity} and Corollary~\ref{cor:dichotomy}.  Second, Kalaj proved that a harmonic
quasiconformal homeomorphism from a real unit ball onto a bounded
$C^1$-Dini domain is bi-Lipschitz once it is known to be a local
$C^1$-diffeomorphism \cite[Theorem~22]{Kalaj2026}.  The missing local
nondegeneracy is supplied by Theorem~\ref{thm:lewy} in every even real dimension.

\begin{corollary}
\label{cor:bilipschitz-intro}
Let $\Omega\subset\C^n\cong\R^{2n}$ be a bounded domain with $C^1$-Dini
boundary.  Every quasiconformal pluriharmonic homeomorphism
\[
      F\colon\B^n\longrightarrow\Omega
\]
is bi-Lipschitz.  In particular, every quasiconformal pluriharmonic
self-homeomorphism of $\B^n$ is bi-Lipschitz.
\end{corollary}

The proof has four independent components.  Degeneracy of $dF_p$ produces a
real scalar projection $u=\ell\circ F$ with $du(p)=0$.  Pluriharmonicity gives
$u=\Repart q$ and $dq(p)=0$.  Local injectivity and invariance of domain force
the level set $\{u=u(p)\}$ to be a locally flat hyperplane pullback.  Finally,
Theorem~\ref{thm:phase-criterion} excludes such a level set: local conical structure
and Alexander duality would make a Milnor fiber rationally acyclic, forcing
the monodromy Lefschetz number to be $1$, whereas A'Campo's theorem makes it
$0$.  The last statement applies without isolatedness or reducedness
assumptions \cite{ACampo1973,GimenezLeNuno2023}.

The paper is organized as follows.  Section~\ref{sec:reduction} gives the scalar
reduction.  Section~\ref{sec:local-topology} records the local-homology and duality
input.  Section~\ref{sec:milnor} carefully matches the spherical Milnor fibration
with the local monodromy used in A'Campo's theorem, including the nonisolated
case.  Section~\ref{sec:criterion} proves the phase-hypersurface criterion, and
Sections~\ref{sec:lewy-proof} and \ref{sec:applications} prove the mapping theorem and its
structural and quantitative consequences.  We stress that the paper addresses
Problem~3.10(a) only; the boundary-extension question in Problem~3.10(b) is a
different problem.
\vskip .10in
\section{Pluriharmonic reduction to a scalar holomorphic germ}
\label{sec:reduction}

We first isolate the only point at which pluriharmonicity enters the proof.

\begin{lemma}
\label{lem:primitive}
Let $U\subset\C^n$ be a ball and let $u\in C^2(U,\R)$ be pluriharmonic.  Then
there is a holomorphic function $q$ on $U$ such that $u=\Repart q$.  The
function $q$ is unique up to an imaginary constant.
\end{lemma}

\begin{proof}
The $(1,0)$-form
\[
       \alpha:=2\partial u
       =2\sum_{j=1}^n u_{z_j}\,dz_j
\]
has holomorphic coefficients because $\overline\partial\partial u=0$, and it
is $\partial$-closed.  The holomorphic Poincar\'e lemma on the ball gives a
holomorphic function $q$ with $dq=\alpha$.  Since
\[
       d(\Repart q)=\Repart(dq)=\Repart(2\partial u)=du,
\]
$u-\Repart q$ is constant; changing the real part of the additive constant in
$q$ gives the identity.  If two holomorphic functions have the same real
part, their difference is a purely imaginary constant.
\end{proof}

\begin{lemma}
\label{lem:real-complex}
Let $q$ be holomorphic near $p\in\C^n$.  Then
\[
       d(\Repart q)_p=0\Longleftrightarrow dq_p=0.
\]
More precisely, in standard coordinates,
\begin{equation}\label{eq:gradient-identity}
   |\nabla \Repart q|^2=\sum_{j=1}^n\left|\frac{\partial q}{\partial z_j}\right|^2.
\end{equation}
\end{lemma}

\begin{proof}
Write $q=u+iv$.  The Cauchy--Riemann equations in each variable give
\[
   q_{z_j}=u_{x_j}-iu_{y_j},
\]
which proves \eqref{eq:gradient-identity} and hence the equivalence.  One may
also argue invariantly: if $d(\Repart q)_p(v)=0$ for every real tangent vector
$v$, then applying the same identity to $iv$ shows that both the real and the
imaginary parts of $dq_p(v)$ vanish.
\end{proof}

The following proposition is the mapping-theoretic core.  Notice that the
homeomorphism is one of ambient pairs, a fact stronger than the intrinsic
local-homology conclusion used later.

\begin{proposition}
\label{prop:reduction}
Let $F\colon\Omega\subset\C^n\to\C^n$ be pluriharmonic and suppose that $F$
is one-to-one on a neighborhood of $p\in\Omega$.  If $\Jac F(p)=0$, then there
exist a nonzero real linear functional $\ell\colon\C^n\to\R$, a ball
$U\ni p$, and a nonconstant holomorphic function $q$ on $U$ such that
\begin{enumerate}[label=\textup{(\roman*)}]
\item $q(p)=0$ and $dq(p)=0$;
\item $\{\Repart q=0\}=F^{-1}(F(U)\cap H)$ in $U$, where
\[
       H:=\{w:\ell(w)=\ell(F(p))\};
\]
\item the germ of the ambient pair
\[
   \bigl(U,\{\Repart q=0\}\bigr)
\]
at $p$ is homeomorphic to the germ of
$\bigl(\R^{2n},\R^{2n-1}\bigr)$ at the origin.
\end{enumerate}
In particular, $\{\Repart q=0\}$ has the local rational and integral homology
of $\R^{2n-1}$ at $p$.
\end{proposition}

\begin{proof}
Shrink to an open ball $U$ on which $F$ is injective.  Since the real linear
map $dF_p\colon\R^{2n}\to\R^{2n}$ is singular, there is a nonzero real
covector $\ell$ such that $\ell\circ dF_p=0$.  Set
\[
       u(z):=\ell(F(z))-\ell(F(p)).
\]
Every real linear combination of the real coordinate functions of $F$ is
real-valued pluriharmonic, so $u$ is pluriharmonic and $du(p)=0$.

The function $u$ is not constant.  Indeed, invariance of domain implies that
$F(U)$ is open in $\R^{2n}$; if $u$ were constant, $F(U)$ would be contained
in the affine hyperplane $H$, which has empty interior.  After shrinking $U$
if needed, Lemma~\ref{lem:primitive} gives a nonconstant holomorphic $q$ with
$u=\Repart q$ and $q(p)=0$.  By Lemma~\ref{lem:real-complex}, $dq(p)=0$.

Again by invariance of domain, $F\colon U\to F(U)$ is a homeomorphism onto an
open set.  Moreover,
\[
   \{z\in U:\Repart q(z)=0\}
      =F^{-1}(F(U)\cap H).
\]
Thus $F$ is a homeomorphism of pairs
\[
 \bigl(U,\{\Repart q=0\}\bigr)
   \longrightarrow
 \bigl(F(U),F(U)\cap H\bigr).
\]
Since $F(U)$ is a neighborhood of $F(p)$ and $H$ is an affine hyperplane
through $F(p)$, the last pair is the standard locally flat hyperplane pair.
The local-homology assertion follows.
\end{proof}
\vskip .10in
\section{Local homology, links, and complementary regions}
\label{sec:local-topology}

We collect the topological input with rational coefficients, which is all that
is needed for the monodromy argument.  For background on local conical
structure and local homology of analytic sets, see
\cite{BochnakCosteRoy1998,BurgheleaVerona1972,
GoreskyMacPherson1988,Hardt1975,Lojasiewicz1964}.

Let $X\subset\R^N$ be a real analytic germ at $0$.  For sufficiently small
admissible $\varepsilon>0$, its link is
\[
       L_\varepsilon(X):=X\cap S_\varepsilon^{N-1}.
\]
The local conical structure theorem gives a homeomorphism from a sufficiently
small representative of $X$ to the cone on $L_\varepsilon(X)$, taking $0$ to
the cone vertex.

\begin{lemma}
\label{lem:local-link}
For every sufficiently small admissible $\varepsilon$ and every integer $j$,
\begin{equation}\label{eq:local-link}
 H_j(X,X\setminus\{0\};\Q)
   \cong \wtH_{j-1}\bigl(L_\varepsilon(X);\Q\bigr).
\end{equation}
Here the left-hand side is computed in a sufficiently small representative of
the germ.
\end{lemma}

\begin{proof}
By local conical structure, the relevant pair is homeomorphic to
$(CL,CL\setminus\{v\})$, where $L=L_\varepsilon(X)$ and $v$ is the cone
point.  The cone is contractible and its punctured cone deformation retracts
onto $L$.  The long exact sequence of the pair gives
\[
 H_j(CL,CL\setminus\{v\};\Q)
    \cong \wtH_{j-1}(L;\Q).
\]
\end{proof}

\begin{corollary}
\label{cor:homology-link}
If $X$ has at $0$ the local rational homology of $\R^d$, then every
sufficiently small admissible link of $X$ has the rational homology of
$S^{d-1}$.
\end{corollary}

\begin{proof}
This is immediate from Lemma~\ref{lem:local-link}.
\end{proof}

The next lemma is an Alexander-duality statement.  Its conclusion concerns
each complementary component, not merely the homology of their union.

\begin{lemma}
\label{lem:alexander}
Let $L\subset S^m$, $m\ge1$, be a compact locally contractible subset having
the rational homology of $S^{m-1}$.  Then $S^m\setminus L$ has exactly two
connected components, and each component is rationally acyclic.
\end{lemma}

\begin{proof}
Alexander duality over $\Q$ gives, for every $i$,
\begin{equation}\label{eq:alexander}
  \wtH_i(S^m\setminus L;\Q)
       \cong \wtH^{m-i-1}(L;\Q);
\end{equation}
see \cite{Bredon1993,Hatcher2002}.  The universal coefficient theorem and the
homology assumption show that the right-hand side is $\Q$ for $i=0$ and zero
for $i>0$.  Since the complement is an open subset of a manifold, it is
locally path connected.  Hence $\wtH_0(S^m\setminus L;\Q)\cong\Q$ says that
the complement has exactly two components, say $V_1,V_2$.  For $i>0$,
\[
 H_i(S^m\setminus L;\Q)
    \cong H_i(V_1;\Q)\oplus H_i(V_2;\Q)=0.
\]
Each $V_k$ is connected, so its reduced $H_0$ also vanishes.  Thus both
components are rationally acyclic.
\end{proof}

\begin{remark}\label{rem:no-schoenflies}
No generalized Schoenflies theorem is used.  The link need not be an unknotted
sphere, or even a topological manifold.  Its rational homology and Alexander
duality already give exactly the information required below.
\end{remark}
\vskip .10in
\section{Milnor fibrations and the A'Campo obstruction}
\label{sec:milnor}

Let $q\colon(\C^n,0)\to(\C,0)$ be a nonconstant holomorphic germ.  For all
sufficiently small $\varepsilon>0$, Milnor's theorem gives a smooth locally
trivial fibration
\begin{equation}\label{eq:spherical-milnor}
 \Phi_q:=\frac{q}{|q|}\colon
 S_\varepsilon^{2n-1}\setminus K_q\longrightarrow S^1,
 \quad
 K_q:=q^{-1}(0)\cap S_\varepsilon^{2n-1}.
\end{equation}
This spherical fibration exists for arbitrary holomorphic function germs; an
isolated critical point is not required \cite{Dimca1992,Milnor1968}.  We denote a fiber
by $P_q$ and its geometric monodromy by $h_q\colon P_q\to P_q$.

There is a technical point that is important here.  A'Campo's theorem is often
phrased for the local tube fibration, while \eqref{eq:spherical-milnor} is the
spherical version used in our sign decomposition.  These fibrations have the
same fiber homotopy type and monodromy.  The equivalence in the nonisolated
case follows from the stratified fibration theory of Hamm and L\^e
\cite{Hamm1971,HammLe1973}; see also
\cite[Section~1]{GimenezLeNuno2023} for an explicit modern account.  Thus the
monodromy in \eqref{eq:spherical-milnor} is precisely the local monodromy to
which A'Campo's theorem applies.

The fiber has finite CW type, and its rational homology is finite dimensional.
We use the Lefschetz number
\[
  \Lambda(h_q)
     :=\sum_{j\ge0}(-1)^j
       \operatorname{tr}\bigl((h_q)_*\mid H_j(P_q;\Q)\bigr).
\]

\begin{theorem}[A'Campo]\label{thm:acampo}
Let $q\colon(\C^n,0)\to(\C,0)$ be holomorphic.  If
$q\in\m_{\C^n,0}^{,2}$, equivalently $q(0)=0$ and $dq(0)=0$, then
\begin{equation}\label{eq:acampo-zero}
       \Lambda(h_q)=0.
\end{equation}
\end{theorem}

The theorem in this generality is due to A'Campo
\cite[Th\'eor\`eme~1 bis]{ACampo1973}.  L\^e proved the stronger fixed-point-free
monodromy statement for a smooth ambient space \cite{Le1975}; the modern
theorem of Gim\'enez Conejero, L\^e, and Nu\~no-Ballesteros applies even to a
complex analytic ambient germ \cite[Theorems~0.1--0.2]{GimenezLeNuno2023}.
These results explicitly allow nonisolated critical loci.

\begin{lemma}
\label{lem:acyclic-lefschetz}
If $P$ is a nonempty connected space of finite homological type and
$\wtH_j(P;\Q)=0$ for all $j\ge0$, then every self-homeomorphism
$h\colon P\to P$ satisfies
\(
       \Lambda(h)=1.
\)
\end{lemma}

\begin{proof}
The only nonzero rational homology group is $H_0(P;\Q)\cong\Q$.  Every
self-map of a connected space induces the identity on $H_0$, so its Lefschetz
number is $1$.
\end{proof}

The bridge between phase hypersurfaces and Milnor fibers is the following
elementary observation.

\begin{lemma}
\label{lem:sign-fiber}
For $\theta\in\R$, let
\[
 U_{\theta,+}:=
 \{z\in S_\varepsilon^{2n-1}:\Repart(e^{-i\theta}q(z))>0\}.
\]
If
\[
 I_{\theta,+}:=
 \{e^{it}\in S^1:\theta-\tfrac\pi2<t<\theta+\tfrac\pi2\},
\]
then
\[
       U_{\theta,+}=\Phi_q^{-1}(I_{\theta,+}).
\]
Consequently, $U_{\theta,+}\simeq P_q$.  The analogous statement holds for
the negative sign region.
\end{lemma}

\begin{proof}
For $z\notin K_q$,
\[
 \Repart(e^{-i\theta}q(z))>0
\Longleftrightarrow
 \Repart\!\left(e^{-i\theta}\frac{q(z)}{|q(z)|}\right)>0,
\]
which is equivalent to $\Phi_q(z)\in I_{\theta,+}$.  The restriction of the
fiber bundle \eqref{eq:spherical-milnor} to the contractible open semicircle
$I_{\theta,+}$ is fiber-homotopy trivial (indeed smoothly trivial), so its
total space is homotopy equivalent to $P_q$.
\end{proof}
\vskip .10in
\section{The phase-hypersurface criterion}
\label{sec:criterion}

We now prove the singularity theorem announced in the introduction.

\begin{proof}[Proof of Theorem~\ref{thm:phase-criterion}]
The implication (i)$\Rightarrow$(ii) follows from
Lemma~\ref{lem:real-complex} applied to $e^{-i\theta}q$: if $dq(0)\ne0$, then
$d\Repart(e^{-i\theta}q)_0\ne0$ for every $\theta$, and the real-analytic
implicit function theorem gives a smooth hypersurface.  The implication
(ii)$\Rightarrow$(iii) is immediate.

It remains to prove (iii)$\Rightarrow$(i).  Suppose that
\eqref{eq:local-Q-homology} holds for a phase $\theta$ and assume, toward a
contradiction, that $dq(0)=0$.  Multiplication by $e^{-i\theta}$ does not
change criticality or monodromy, so we may replace $q$ by
$e^{-i\theta}q$ and assume $\theta=0$.  Put
\[
       X:=\{\Repart q=0\}.
\]

Choose $\varepsilon>0$ sufficiently small so that the local conical structure
of $X$, the local-homology link description, and the Milnor fibration all hold
on $S_\varepsilon^{2n-1}$.  Such radii form a cofinal set at the origin by the
subanalytic conical structure theorem and Milnor's fibration theorem.  Let
\[
       L:=X\cap S_\varepsilon^{2n-1}.
\]
By Corollary~\ref{cor:homology-link}, $L$ has the rational homology of
$S^{2n-2}$.  The link is compact and semianalytic, hence triangulable and
locally contractible.  Applying Lemma~\ref{lem:alexander} in
$S_\varepsilon^{2n-1}$, we conclude that its complement has exactly two
connected components and that each is rationally acyclic.

The complement is the disjoint union
\[
 S_\varepsilon^{2n-1}\setminus L=U_+\sqcup U_-,
 \quad
 U_\pm:=\{z\in S_\varepsilon^{2n-1}:\pm\Repart q(z)>0\}.
\]
Both sign regions are nonempty.  Indeed, $u=\Repart q$ is a nonconstant real
harmonic function with $u(0)=0$.  If $u\ge0$ on the sphere, the minimum
principle first gives $u\ge0$ in the ball and then, from the interior minimum
$u(0)=0$, forces $u$ to be constant.  The same argument applied to $-u$
excludes $u\le0$ on the sphere.

Because the complement has exactly two components, and $U_+$ and $U_-$ are
disjoint nonempty open subsets whose union is the complement, each $U_\pm$ is
one full component.  Hence
\begin{equation}\label{eq:sign-acyclic}
       \wtH_j(U_+;\Q)=0\quad(j\ge0).
\end{equation}
By Lemma~\ref{lem:sign-fiber}, $U_+\simeq P_q$, so $P_q$ is nonempty, connected,
and rationally acyclic.  Thus Lemma~\ref{lem:acyclic-lefschetz} gives
\[
       \Lambda(h_q)=1.
\]
On the other hand, $q(0)=0$ and $dq(0)=0$, so $q\in\m^2$ and A'Campo's
theorem gives $\Lambda(h_q)=0$.  This contradiction proves $dq(0)\ne0$.
\end{proof}

Several formulations follow immediately.

\begin{corollary}
\label{cor:formulations}
For a nonconstant holomorphic germ $q\colon(\C^n,0)\to(\C,0)$, the following
are equivalent:
\begin{enumerate}[label=\textup{(\roman*)}]
\item $dq(0)\ne0$;
\item $\{\Repart q=0\}$ is a smooth real-analytic hypersurface germ at $0$;
\item $\{\Repart q=0\}$ is a topological $(2n-1)$-manifold germ at $0$;
\item $\{\Repart q=0\}$ is a rational homology $(2n-1)$-manifold at $0$.
\end{enumerate}
If these conditions hold, the hypersurface is locally flat in $\R^{2n}$; if
$dq(0)=0$, it is not even a rational homology manifold at $0$.
\end{corollary}

\begin{proof}
The implications
$\textup{(i)}\Rightarrow\textup{(ii)}\Rightarrow
 \textup{(iii)}\Rightarrow\textup{(iv)}$
are standard, while $\textup{(iv)}\Rightarrow\textup{(i)}$ is
\cref{thm:phase-criterion}.  A smooth hypersurface is locally flat by the
implicit function theorem.
\end{proof}

\begin{corollary}
\label{cor:every-phase}
If $dq(0)=0$, then for every real line $L\subset\C$ through $0$, the inverse
image germ $q^{-1}(L)$ fails to be a rational homology $(2n-1)$-manifold at
$0$.
\end{corollary}

\begin{proof}
Every real line through $0$ is
$e^{i(\theta+\pi/2)}\R$, and its inverse image is
$X_\theta(q)$.  Apply \cref{thm:phase-criterion}.
\end{proof}

\begin{example}[The quadratic model]\label{ex:quadratic}
For
\[
       q(z)=z_1^2+\cdots+z_n^2,
       \quad z=x+iy,
\]
one has $\Repart q=|x|^2-|y|^2$.  Hence the link of
$\{\Repart q=0\}$ on $S_\varepsilon^{2n-1}$ is
\[
 S^{n-1}_{\varepsilon/\sqrt2}\times
 S^{n-1}_{\varepsilon/\sqrt2}.
\]
For $n\ge2$ this has nontrivial middle homology and is not a homology
$S^{2n-2}$; for $n=1$ it consists of four points rather than two.  This model
exhibits concretely the homology detected by the proof.  The general argument
shows that some homological obstruction persists for every critical
holomorphic germ, even when its critical locus is nonisolated or its zero
divisor is nonreduced.
\end{example}
\vskip .10in
\section{The pluriharmonic Lewy theorem and the branch set}
\label{sec:lewy-proof}

\begin{proof}[Proof of Theorem~\ref{thm:lewy}]
The implication (ii)$\Rightarrow$(iii) is the real inverse function theorem;
pluriharmonic mappings are real analytic.  The implication
(iii)$\Rightarrow$(i) is immediate.

For (i)$\Rightarrow$(ii), assume that $F$ is injective near $p$ and suppose
that $\Jac F(p)=0$.  By Proposition~\ref{prop:reduction}, there is a nonconstant
holomorphic $q$ near $p$ such that $dq(p)=0$ while
$\{\Repart q=0\}$ is a locally flat hypersurface germ.  After translating
$p$ to $0$, this contradicts Corollary~\ref{cor:formulations}.  Therefore
$\Jac F(p)\ne0$.
\end{proof}

The natural local noninjectivity set of a continuous map is
\[
 \mathcal B_{\mathrm{inj}}(F):=
 \{p\in\Omega:F\text{ is not injective on any neighborhood of }p\}.
\]
For a general smooth mapping, criticality and failure of local injectivity
need not agree.  They do agree for pluriharmonic mappings.

\begin{theorem}
\label{thm:branch-identity}
For every pluriharmonic map $F\colon\Omega\subset\C^n\to\C^n$,
\begin{equation}\label{eq:branch-identity}
       \mathcal B_{\mathrm{inj}}(F)
       =\Crit(F):=\{p\in\Omega:\Jac F(p)=0\}.
\end{equation}
In particular, the local univalence locus is precisely the regular locus of
$F$.
\end{theorem}

\begin{proof}
By Theorem~\ref{thm:lewy}, a locally injective point is regular, which gives
$\Crit(F)\subset\mathcal B_{\mathrm{inj}}(F)$.  Conversely, the inverse
function theorem makes $F$ locally injective at every point where
$\Jac F\ne0$.  Taking the contrapositive gives the reverse inclusion.
\end{proof}

Since $F$ is real analytic, so is $\Jac F$.  The identity
\eqref{eq:branch-identity} therefore gives a useful dichotomy.

\begin{corollary}
\label{cor:dichotomy}
Assume that $\Omega$ is connected.  Exactly one of the following occurs:
\begin{enumerate}[label=\textup{(\roman*)}]
\item $\Jac F\equiv0$, and $F$ is locally injective nowhere;
\item $\mathcal B_{\mathrm{inj}}(F)$ is a proper closed real-analytic subset
of $\Omega$ of real dimension at most $2n-1$, hence has empty interior and
$2n$-dimensional Lebesgue measure zero.  On its open dense complement, $F$ is
a local real-analytic diffeomorphism.
\end{enumerate}
\end{corollary}

\begin{proof}
If $\Jac F\equiv0$, the first alternative follows from
Theorem~\ref{thm:branch-identity}.  Otherwise the zero set of the nontrivial real
analytic function $\Jac F$ is a proper real-analytic subset of codimension at
least one, and therefore has empty interior and measure zero; see, for example,
\cite{KrantzParks2002}.  Apply Theorem~\ref{thm:branch-identity} again.
\end{proof}

\begin{corollary}
\label{cor:global-diffeo}
If $F\colon\Omega\to\C^n$ is injective and pluriharmonic, then
$\Jac F$ is nowhere zero and $F$ is a real-analytic diffeomorphism from
$\Omega$ onto the open set $F(\Omega)$.
\end{corollary}

\begin{proof}
Apply Theorem~\ref{thm:lewy} at every point.  Invariance of domain makes $F(\Omega)$
open and $F$ a homeomorphism onto its image.  The local real-analytic inverses
given by the inverse function theorem agree with the global inverse.
\end{proof}

\begin{corollary}
\label{cor:orientation}
If $\Omega$ is connected and
$F\colon\Omega\to F(\Omega)\subset\C^n$ is a pluriharmonic homeomorphism,
then $\Jac F$ has a constant nonzero sign on $\Omega$.
\end{corollary}

\begin{proof}
By Corollary~\ref{cor:global-diffeo}, $\Jac F$ is continuous and nowhere zero.  Its
sign is constant on the connected domain $\Omega$.
\end{proof}
\vskip .10in
\section{A quantitative application to quasiconformal pluriharmonic maps}
\label{sec:applications}

We now prove Corollary~\ref{cor:bilipschitz-intro}.  Recall that a bounded domain has
$C^1$-Dini boundary if, in finitely many boundary charts, it is represented by
$C^1$ graph functions whose gradients have a common modulus of continuity
$\omega$ satisfying
\[
       \int_0^{r_0}\frac{\omega(t)}{t}\,dt<\infty.
\]
This is the natural regularity at which regularized-distance barriers give
sharp two-sided boundary-distance control; see
\cite{Kalaj2015,Lieberman1985}.

Kalaj's theorem  \cite[Theorem~22]{Kalaj2026} is stated in every real dimension: if
$f\colon B^m\to D\subset\R^m$ is a harmonic quasiconformal homeomorphism onto
a bounded $C^1$-Dini domain and $f$ is a local $C^1$-diffeomorphism in $B^m$,
then $f$ is bi-Lipschitz.  In the pluriharmonic
setting, Theorem~\ref{thm:lewy} removes the local-diffeomorphism hypothesis.

\begin{proof}
View $\B^n$ and $\Omega$ as domains in $\R^{2n}$.  Since $F$ is a
homeomorphism, it is locally injective.  Theorem~\ref{thm:lewy} gives
\[
       \Jac F(z)\ne0\quad(z\in\B^n),
\]
so $F$ is a local real-analytic, hence local $C^1$, diffeomorphism.  Each real
coordinate of a pluriharmonic mapping is harmonic.  Kalaj's theorem therefore
applies with $m=2n$ and yields the bi-Lipschitz conclusion.
\end{proof}

\begin{remark}
\label{rem:distortion}
For general harmonic quasiconformal mappings in $\R^m$, pointwise Jacobian
nondegeneracy is not automatic.  A known sufficient condition is the
small-distortion bound $K_O<3^{m-1}$
\cite{BozinMateljevic2015}; see also \cite{AstalaManojlovic2015} for related
higher-dimensional extensions of Pavlovi\'c-type regularity.  In the
pluriharmonic class, Theorem~\ref{thm:lewy} supplies nondegeneracy with no smallness
condition on the quasiconformal distortion.  Thus
Corollary~\ref{cor:bilipschitz-intro} extends the complex two-dimensional conclusion
in \cite[Corollary~24]{Kalaj2026} to every $n\ge1$.
\end{remark}
\vskip .10in
\section{Scope of the monodromy argument}
\label{sec:scope}

We finish by recording why the argument is dimension-free and where its
hypotheses enter.

For an isolated hypersurface singularity, the Milnor fiber is homotopy
equivalent to a bouquet of middle-dimensional spheres, and one could reach a
contradiction through the Milnor number.  The scalar germ produced by
Proposition~\ref{prop:reduction}, however, has no reason to possess an isolated critical
point.  A'Campo's theorem is exactly suited to this situation: membership in
$\m^2$, rather than isolatedness, forces the monodromy Lefschetz number to
vanish.

If $q$ has repeated factors, its Milnor fiber may be disconnected and the
binding may be singular.  The proof does not separate reduced and nonreduced
cases.  Under the hypothetical homology-manifold condition, Alexander duality
forces each sign region to be connected and rationally acyclic; the sign-region
bundle then forces the same properties on the Milnor fiber.  This is precisely
what contradicts A'Campo's theorem.

Under the local homology hypothesis, the link
\[
 L=\{\Repart q=0\}\cap S^{2n-1}_\varepsilon
\]
is a rational homology $S^{2n-2}$.  Alexander duality turns its complement
into a homological $S^0$: two rationally acyclic components.  A sign component
is the Milnor fibration over a contractible semicircle.  After that passage,
all link geometry is compressed into the scalar contradiction
\[
        1=\Lambda(h_q)=0.
\]
The left equality is topological; the right equality is holomorphic.

The result proves the local equivalence between univalence and Jacobian
nondegeneracy for equidimensional pluriharmonic maps, together with the stated
global and quasiconformal consequences.  It does not assert a Lewy theorem for
arbitrary harmonic maps, which is false in real dimensions at least three, and
it does not address the separate boundary-extension question in
Hengartner's Problem~3.10(b).

\medskip
\noindent\textbf{Acknowledgements.}
D. Zhong was partially supported by the \textit{Guangdong Basic and Applied Basic Research Foundation} under Grant nos. 2022A1515110967 and 2023A1515011809 of the P. R. China.
Z.-G. Wang was partially supported by the Key Project of the Education
Department of Hunan Province under Grant no.~25A0668 and by the Natural
Science Foundation of Changsha under Grant no.~kq2502003.  

\medskip
\noindent\textbf{Author contributions.}
Both authors contributed equally to this work.

\medskip
\noindent\textbf{Conflict of interest.}
The authors declare no conflict of interest.

\medskip
\noindent\textbf{Data availability.}
No datasets were generated or analyzed in this work.

\end{document}